\documentclass[11pt,reqno,bold]{paper}
\pdfoutput=1
\usepackage{geometry}
\usepackage{amsmath}
\usepackage{amssymb}
\usepackage{amsfonts}
\usepackage{mathrsfs}
\usepackage{asymptote}

\usepackage{todonotes}

\usepackage[round]{natbib}
\usepackage[algoruled,noline]{algorithm2e}
\usepackage{enumerate}
\usepackage{comment}

\usepackage{amsthm}
\newtheorem{result}{Main result}
\newtheorem{theorem}{Theorem}[section]

\newtheorem{lemma}[theorem]{Lemma}
\newtheorem{corollary}[theorem]{Corollary}
\theoremstyle{definition}

\theoremstyle{remark}
\newtheorem{remark}[theorem]{Remark}

\usepackage{color}
\definecolor{darkblue}{rgb}{0.0,0.0,0.4}
\usepackage[pdfauthor={Tsogtgerel Gantumur},pdftitle={Adaptive mixed finite element method for the Stokes problem},pdfsubject={adaptive methods, convergence rate, Stokes problem},colorlinks,citecolor=darkblue,linkcolor=darkblue,urlcolor=darkblue]{hyperref}
\newcommand{\N}{\mathbb{N}}
\newcommand{\Z}{\mathbb{Z}}
\newcommand{\R}{\mathbb{R}}
\newcommand{\Pol}{\mathbb{P}}
\newcommand{\tstA}{\mathscr{A}}
\newcommand{\tstP}{\mathscr{P}}

\newcommand{\tstO}{\mathscr{O}}
\newcommand{\eps}{\varepsilon}
\newcommand{\vol}{\mathrm{vol}}

\newcommand{\diam}{\mathrm{diam}}
\newcommand{\osc}{\mathrm{osc}}

\newcommand{\intr}{\mathrm{int}}

\newcommand{\conf}{\mathrm{conf}}

\newcommand{\refine}{\mathrm{refine}}

\newcommand{\shape}{\sigma_{\textrm{s}}}
\newcommand{\grade}{\sigma_{\textrm{g}}}

\excludecomment{comment:quasi}

\title{On the convergence theory of adaptive mixed finite element methods for the Stokes problem}
\author{Tsogtgerel Gantumur}
\institution{McGill University}
\date{\today}                                           % Activate to display a given date or no date

\begin{document}

%%%%%%%%%%%%%%%%%%%%%%%%%%%%%%%%%%%%%%%%%%%%%%%%%%%%%%%%%%%%%%%%%%%%%%%%%%%%%%
\maketitle

%%%%%%%%%%%%%%%%%%%%%%%%%%%%%%%%%%%%%%%%%%%%%%%%%%%%%%%%%%%%%%%%%%%%%%%%%%%%%%
\begin{abstract}
In this paper, we establish a {\em conditional} optimality result for an adaptive mixed finite element method for the stationary Stokes problem
discretized by the standard Taylor-Hood elements, under the assumption of the so-called general quasi-orthogonality.
Optimality is measured in terms of a modified approximation class defined through the total error,
as is customary since the seminal work of Cascon, Kreuzer, Nochetto and Siebert.
The second part of the paper is independent of optimality results, and concerns interrelations between the 
modified approximation classes and the standard approximation classes (the latter defined through the energy error).
Building on the tools developed in the papers of Binev, Dahmen, DeVore, and Petrushev, and of Gaspoz and Morin,
we prove that the modified approximation class coincides with the standard approximation class,
modulo the assumption that the data is regular enough in an appropriate scale of Besov spaces.
\end{abstract}

\tableofcontents

%%%%%%%%%%%%%%%%%%%%%%%%%%%%%%%%%%%%%%%%%%%%%%%%%%%%%%%%%%%%%%%%%%%%%%%%%%%%%%
\section{Introduction}

We consider adaptive mixed finite element methods for the {\em stationary Stokes problem}
\begin{equation}\label{e:Stokes-classical}
\begin{split}
- \Delta u + \nabla p &= f, \\
\nabla\cdot u &= 0,
\end{split}
\end{equation}
discretized by the standard Taylor-Hood elements.
Here $u:\Omega\to\R^n$ is the unknown {\em velocity field}, $p:\Omega\to\R$ is the unknown {\em pressure field},
$f\in L^2(\Omega,\R^n)$ is the given {\em data}, and $\Omega\subset\R^n$ is a bounded polyhedral domain
with Lipschitz boundary.
One can think of the space dimension to be $n=2$ or $n=3$.
We impose the {\em no slip} boundary condition $u|_{\partial\Omega}=0$ on the velocity field,
and in order to ensure uniqueness, we require that the pressure field is of {\em vanishing mean}.

Convergence theory of adaptive finite element methods has been an active field of research especially since the influential paper \cite*{Dorf96}.
A near complete understanding has been achieved for Poisson-type problems, cf. \cite*{MNS00,BDD04,Stev07,CKNS08}.
There is a growing body of literature on adaptive discretization of saddle point problems such as \eqref{e:Stokes-classical},
but the question of convergence rate for adaptive mixed finite element methods with standard Taylor-Hood elements has been entirely open.

Building on the pioneering works \cite*{DHU00,BMN02},
optimal convergence rates have been established for certain elliptic reformulations of the Stokes problem, in \cite*{Kond06,KS08}.
Moreover, for nonconforming discretizations of the Stokes problem, the same question has been investigated in \cite*{BeMa11,HuXu13,CPR13}.
On the other hand, adaptive mixed finite element methods for the Poisson problem have been treated in \cite*{CaHo06,BeMa08,CHX09}.

Getting back to the discussion of adaptive mixed finite element methods with standard Taylor-Hood elements,
the first proof of convergence for such a method was published in \cite*{MSV08}, in which the authors modified the standard {\em a posteriori} error estimator from \cite{Verf89}.
This proof was improved in \cite{Sieb11} to incorporate the  {\em a posteriori} error estimator from \cite{Verf89} into the analysis.
Note that the aforementioned results do not provide information about the {\em rate} of convergence.

In the current work, 
under the assumption of the so-called {\em general quasi-orthogonality},
we establish bounds on the convergence rates of adaptive mixed finite element methods for the Stokes problem discretized by standard Taylor-Hood elements,
and show that these bounds are in a certain sense optimal.
This was motivated by the conceptual understanding of the role played by the general quasi-orthogonality in the analysis of adaptive methods, cf. \cite*{FFP12,CFPP14}.
What the concept of general quasi-orthogonality provides is a framework to potentially exploit the Galerkin orthogonality for non-coercive or strongly nonsymmetric problems,
a bottleneck that has been faced by researchers for some time.

In order to discuss our other results, we need to fix some notation and terminologies.
The solution $(u,p)\in H^1(\Omega,\R^n)\times L^2(\Omega)$ of \eqref{e:Stokes-classical} is said to be a member of the {\em standard approximation class} $\tstA^s$
if there exists a sequence of conforming triangulations $P_1,P_2,\ldots$ of $\Omega$, 
obtained from a fixed initial triangulation $P_0$ by applications of newest vertex bisections,
such that $\#P_N-\#P_0\leq N$ and
\begin{equation}\label{e:standard-approx-class-error}
\|u-u_N\|_{H^1(\Omega,\R^n)} + \|p-p_N\|_{L^2(\Omega)} \leq C N^{-s} ,
\qquad \textrm{for all}\quad N ,
\end{equation}
where $(u_N,p_N)$ is the Galerkin approximation of $(u,p)$ from the Taylor-Hood finite element space (of degree $d$ for the velocity field and degree $d-1$ for the pressure field) defined over the triangulation $P_N$, and $C$ is a constant independent of $N$.
Suppose that we have an adaptive algorithm (based on the same Taylor-Hood spaces and newest vertex bisections) that takes the data $f\in L^2(\Omega,\R^n)$ and the initial triangulation $P_0$ of $\Omega$ as its input, 
and produces the sequence of triangulations $P_1,P_2,\ldots$, and the corresponding Galerkin solutions $(u_k,p_k)$ for $k=0,1,\ldots$.
Then it stands to reason to say that the algorithm converges at the {\em optimal rate} if 
\begin{equation}
\|u-u_k\|_{H^1(\Omega,\R^n)} + \|p-p_k\|_{L^2(\Omega)} \leq C (\#P_k-\#P_0)^{-s} ,
\qquad \textrm{for all}\quad k ,
\end{equation}
whenever $(u,p)\in\tstA^s$ for some $s>0$.
For the Poisson equation $\Delta u = g$, such an optimality result was obtained in \cite{Stev07}, for an adaptive algorithm that uses an inner loop for resolving the data $g$.
It was later discovered in \cite{CKNS08} that if we do not add an inner iteration, 
the algorithm remains optimal {\em provided} that we modify the approximation class to include in its definition a measure of resolution of $g$.
The analogue of this modification in our setting is as follows.
The solution $(u,p)\in H^1(\Omega,\R^n)\times L^2(\Omega)$ of \eqref{e:Stokes-classical} is said to be a member of the {\em modified approximation class} $\tstA_*^s$ 
if there exists a sequence of conforming triangulations $P_1,P_2,\ldots$ of $\Omega$, 
obtained from a fixed initial triangulation $P_0$ by applications of newest vertex bisections,
such that $\#P_N-\#P_0\leq N$ and
\begin{equation}\label{e:modified-approx-class-error}
\|u-u_N\|_{H^1(\Omega,\R^n)} + \|p-p_N\|_{L^2(\Omega)} + \osc_N(f) \leq C N^{-s} ,
\qquad \textrm{for all}\quad N ,
\end{equation}
where $(u_N,p_N)$ is the Galerkin approximation of $(u,p)$ from the Taylor-Hood finite element space defined over the triangulation $P_N$,
and $\osc_N(f)$ is the so-called {\em oscillation} term, which depends only on $P_N$ and $f$.
Note that $\tstA_*^s$ is indeed a space of pairs $(u,p)$, because $f$ is completely determined by $u$ and $p$.
Note also that $\tstA_*^s\subset\tstA^s$, since the membership of $\tstA_*^s$ has the extra requirement that the oscillation $\osc_N(f)$ is suitably reduced as $N$ grows.
Although reducing the oscillation is by no means among our initial ambitions, it turns out that $\tstA_*^s$ is completely natural from the perspective of adaptive methods.
In particular, the quantity on the left hand side of \eqref{e:modified-approx-class-error}, called the {\em total error}, is equivalent to the error estimator,
and so an algorithm that only ``sees'' the error estimator will have to reduce the oscillation anyway.
With respect to the approximation classes $\tstA_*^s$, we have the following {\em conditional} optimality result, which will be proved in Section \ref{s:convergence}.
The aforementioned general quasi-orthogonality assumption appears here in \eqref{e:gen-quasorth}.

\begin{result}\label{mr:rate-mod}
Suppose that $(u,p)\in\tstA_*^s$ for some $s>0$, 
and let $P_0,P_1,\ldots$ be the sequence of triangulations generated by the adaptive algorithm defined in Section \ref{s:convergence},
with $(u_k,p_k)$, $k=0,1,\ldots$, being the corresponding Galerkin solutions.
In addition, assume that there exists a constant $c>0$ such that
\begin{equation}\label{e:gen-quasorth}
\begin{split}
\sum_{k=\ell}^{\infty} \Big( \|u_k-u_{k+1}\|_{H^1(\Omega,\R^n)}^2 + \|p_k-p_{k+1}&\|_{L^2(\Omega)}^2 \Big) \\
&\leq 
c \Big( \|u-u_\ell\|_{H^1(\Omega,\R^n)}^2 + \|p-p_\ell\|_{L^2(\Omega)}^2 \Big) ,
\end{split}
\end{equation}
for any integer $\ell$.
Then we have
\begin{equation}\label{e:mr-rate-mod}
\|u-u_k\|_{H^1(\Omega,\R^n)} + \|p-p_k\|_{L^2(\Omega)} + \osc_k(f) \leq C (\#P_k-\#P_0)^{-s} ,
\qquad \textrm{for all}\quad k ,
\end{equation}
where $\osc_k(f)$ denotes the oscillation on the mesh $P_k$.
\end{result}

We have this conditional optimality result for an algorithm that uses the {\em a posteriori} error estimator from \cite{Verf89}.
As a theoretical tool to be employed in the analysis, we introduce a supposedly new {\em a posteriori} error estimator, which also yields optimal algorithms.
For algorithms that use the modified estimator from \cite{MSV08}, we establish geometric error reduction,
but we were unable to prove optimal convergence rates, because of the apparently nonlocal character of the estimator.

There is a remark to be made on the nature of the constant $C$ that appears in \eqref{e:mr-rate-mod}.
From the experience with Poisson-type problems, one would expect that the constant $C$ must be of the form
$C=c|(u,p)|_{\tstA_*^s}$, where $c$ does not depend on $(u,p)$, 
and $|(u,p)|_{\tstA_*^s}$ is the norm of $(u,p)$ in the space $\tstA_*^s$.
In Main result \ref{mr:rate-mod}, however, we do not rule out the possibility that $c$ depends on $(u,p)$.
%This matter calls for further investigation.

%The above optimality theorem hinges on the afore-mentioned general quasi-orthogonality result, which is proved in Section \ref{s:Bessel}.
The final part of the current paper is independent of optimality results, and concerns interrelations between the 
modified approximation classes $\tstA_*^s$ and the standard approximation classes $\tstA^s$.
The pairs $(u,p)$ contained in the gap $\tstA^s\setminus\tstA_*^s$ can in principle be approximated with the rate $s$ by some approximation procedure,
but Main result \ref{mr:rate-mod} cannot guarantee that adaptive finite element methods will converge with the same rate.
Our approach to this problem is to show that $(u,p)$ is {\em not} in the gap, as long as the data $f$ have some regularity in terms of a scale of Besov spaces.
Namely, the following result will be proved in Section \ref{s:approximation}.

\begin{result}\label{mr:Besov-rate}
Let $f\in B^\alpha_{q,q}(\Omega,\R^n)$ for some $0<q<\infty$ and $\alpha\geq\frac{n}q - \frac{n}2$ satisfying $0<\alpha<d-1+\max\{0,\frac1q-1\}$.
Then $(u,p)\in \tstA^s$ implies $(u,p)\in\tstA_*^s$, with $s=\frac{\alpha+1}n$.
\end{result}

\begin{figure}[ht]
\begin{center}
\includegraphics{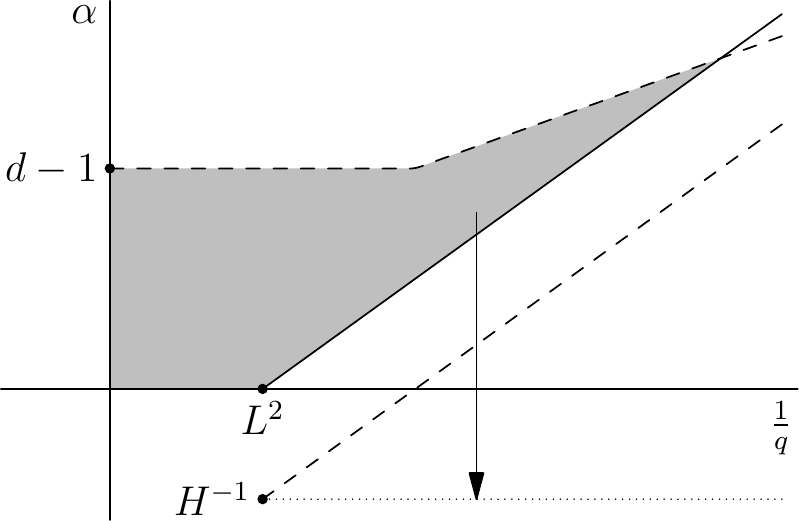}
\end{center}
\caption{The shaded region represents the pairs $(\frac1q,\alpha)$ allowed in Main result \ref{mr:Besov-rate}.
If the function $f\in B^{\alpha}_{q,q}$ was being adaptively approximated in the $H^{-1}$ norm, 
and if $B^{\alpha}_{q,q}\subset H^{-1}$, then we would expect
the rate of convergence to be determined by the vertical offset of the point $(\frac1q,\alpha)$ in relation to the dotted line.
Main result \ref{mr:Besov-rate} basically says that the same convergence rate is restored for the total error as long as $B_{q,q}^\alpha$ is embedded into $L^2$,
and $(u,p)$ is regular enough (to be in $\tstA^s$).}
\end{figure}

To prove this result, we adapt the techniques from \cite*{BDDP02} and \cite{GM13}, 
where the said techniques have been used to show embeddings of the form $B_{q,q}^\alpha\subset\tstA^s$.
Earlier influential works in the same vein include \cite{BirSol67} and \cite{Osw90}.
We also establish embeddings of the form $B_{q,q}^\alpha\subset\tstA^s$ in the context of the Stokes problem, see Theorem \ref{t:direct-H1}.
Note that the regularity of the Stokes problem in the same Besov scale has been studied in \cite{Dahl99c}.

This paper is organized as follows.
In the next section, we will discuss assumptions on the triangulations, the refinement procedures, and the finite element spaces for discretizing the Stokes problem.
Then in Section \ref{s:error-estimators}, we introduce three kinds of {\em a posteriori} error estimators, and establish some of their properties.
The conditional optimality result mentioned above, together with a theorem on geometric error reduction will be proved in Section \ref{s:convergence}.
Finally, in Section \ref{s:approximation} we will deal with interplays between the approximation classes and Besov spaces.

%%%%%%%%%%%%%%%%%%%%%%%%%%%%%%%%%%%%%%%%%%%%%%%%%%%%%%%%%%%%%%%%%%%%%%%%%%%%%%
\section{Discretization of the Stokes problem}
\label{s:disc}

Let $\Omega\subset\R^n$ be a polyhedral domain with Lipschitz boundary, where $n=2$ or $n=3$.
We call a collection $P$ of triangles (or tetrahedra) a {\em partition} of $\Omega$
if  $\overline\Omega=\bigcup_{\tau\in P}\overline\tau$, and
$\tau\cap\sigma=\varnothing$ for any two different $\tau,\sigma\in P$. 
%A partition $P$ is said to be {\em conforming} if any vertex $v$ of a triangle in $P$ is a vertex of all $\tau\in P$ whose closure contains $v$. 
For refining the partitions we use the so called {\em newest vertex bisection} algorithm;
details can be found in \citet*{BDD04,Stev08}.
A partition $P'$ is called a {\em refinement} of $P$ and
denoted $P\preceq P'$ if $P'$
can be obtained by replacing zero or more $\tau\in P$ by its
children, or by a recursive application of this procedure.
Throughout this paper, we only consider {\em conforming} partitions that are refinements of some {\em fixed} conforming partition $P_0$ of
$\Omega$. 
The newest vertex bisection procedure produces {\em shape regular} partitions, meaning that
\begin{equation}\label{e:shape-regular}
\shape=\sup \left\{ \frac {(\diam\,\tau)^n}{\vol(\tau)} : \tau\in P,\,P\in\conf(P_0)
  \right\} < \infty,
\end{equation}
where $\conf(P_0)$ denotes the family of all conforming partitions that are refinements of $P_0$.
This family is \emph{graded} (or \emph{locally quasi-uniform}), in the sense that 
\begin{equation}\label{e:K-mesh}
\grade=\sup \left\{\frac {\diam\,\sigma}{\diam\,\tau}: \sigma,\tau\in
P,\,\overline\sigma\cap\overline\tau\neq\varnothing,\,P\in\conf(P_0)\right\} < \infty.
\end{equation}
Note that the shape regularity and gradedness
together imply \emph{local finiteness}, meaning that the number of triangles
meeting at any given point is bounded by a constant that depends only
on $\shape$, $\grade$, and $n$.

In general, a naive refinement of a conforming partition would produce a non-conforming partition, 
so in order to ensure conformity, one must perform additional refinements. 
This procedure is called {\em completion}, and a quite satisfactory theory of completion has been developed in \cite{BDD04,Stev08}.
We consider the whole process of obtaining a conforming partition from an initial conforming partition as a single refinement step that works in the category of conforming partitions.
Given a partition $P\in\conf(P_0)$ and a set $R\subset P$ of its triangles,
the refinement step produces $P'\in\conf(P_0)$,
such that $P\setminus P'\supseteq R$, i.e., the triangles in $R$ are refined at least once.
Let us denote it by $P'=\refine(P,R)$.
We have the following on its efficiency:
If $\{P_k\}\subset\conf(P_0)$ and $\{R_k\}$ are sequences such that $P_{k+1}=\refine(P_k,R_k)$
and $R_k\subset P_k$ for $k=0,1,\ldots$,
then
\begin{equation}\label{e:complete}
\#P_\ell-\#P_0\leq C_c\sum_{k=0}^{\ell-1} \#R_k,
\qquad \ell=1,2,\ldots,
\end{equation}
where $C_c>0$ is a constant.

Another notion we need is that of \emph{overlay} of partitions:
We assume that there is an operation
$\oplus:\conf(P_0)\times\conf(P_0)\to\conf(P_0)$ satisfying
\begin{equation}\label{e:overlay}
P\oplus Q\succeq P,
\qquad
P\oplus Q\succeq Q,
\qquad
\textrm{and}
\qquad
\#(P\oplus Q)\leq\#P+\#Q-\#P_0,
\end{equation}
for $P,Q\in\conf(P_0)$.
This assumption is verified in \cite{Stev07,CKNS08}, 
where $P\oplus Q$ is taken to be the smallest and common conforming refinement of $P$ and $Q$.

Let $V=H^1_0(\Omega,\R^n)$ and $Q = L^2(\Omega)/\R$, the latter being the space of $L^2$ functions with vanishing mean,
and let $X=V\times Q$ be the Hilbert space equipped with the norm
\begin{equation}
\|(v,q)\|_{V\times Q} = \big( \|v\|_V^2 + \|q\|_Q^2 \big)^{\frac12} .
\end{equation}
We consider the following weak formulation of the Stokes problem \eqref{e:Stokes-classical}: Find $(u,p)\in X$ satisfying
\begin{equation}\label{e:Stokes-weak}
a(u,v) - b(v,p) - b(u,q) = \langle f,v\rangle_{L^2},
\qquad \textrm{for all}\quad
(v,q) \in X ,
\end{equation}
where the bilinear forms $a:V\times V\to\R$ and $b:V\times Q\to\R$ are defined respectively by
\begin{equation}
a(u,v) = \int_\Omega {\textstyle\sum_{i,k}}\partial_i u_k \partial_i v_k,
\qquad \textrm{and} \qquad
b(u,q) = \int_\Omega q{\textstyle\sum_{i}}\partial_iu_i,
\end{equation}
and $\langle\cdot,\cdot\rangle_{L^2}$ denotes the $L^2$ inner product.
%In the following, we will use the shorthand notations $Du:Dv=\sum_{i,k}\partial_i u_k \partial_i v_k$ and $\nabla\cdot u=\sum_{i}\partial_iu_i$.
It is known that for any $f\in L^2(\Omega,\R^n)$, the problem \eqref{e:Stokes-weak} admits a unique solution $(u,p)\in X$, see, e.g., \cite{BrezFort91}.
In fact, the operator $A:X \to X'$ defined by
\begin{equation}\label{e:Stokes-operator}
\langle A(u,p),(v,q) \rangle = a(u,v) - b(v,p) - b(u,q) ,
\qquad
(u,p), (v,q) \in X ,
\end{equation}
is invertible, where $\langle\cdot,\cdot\rangle$ denotes the duality pairing between $X'$ and $X$.
It is easy to see that $A$ is linear, bounded, and self-adjoint.
In particular, by Banach's bounded inverse theorem, the inverse $A^{-1}:X'\to X$ is bounded as well.
In terms of the operator $A$, the Stokes problem \eqref{e:Stokes-weak} can be written as
\begin{equation}\label{e:Stokes-operator-form}
A(u,p) = F ,
\end{equation}
where the linear functional $F\in X'$ is defined by
\begin{equation}\label{e:Stokes-rhs-op}
\langle F,(v,q) \rangle = \langle f,v\rangle_{L^2},
\qquad \textrm{for all}\quad
(v,q) \in X.
\end{equation}

In the discretization of \eqref{e:Stokes-weak}, we will use the classical finite element spaces
introduced by \cite{TaylHood73}.
Given a partition $P\in\conf(P_0)$, we define the discontinuous piecewise polynomial space ${S}_{P}^d$ by
\begin{equation}\label{e:fem-space}
{S}_{P}^d=\left\{u\in L^\infty(\Omega) : u|_{\tau}\in \mathbb{P}_{d}\,\forall\tau\in P\right\},
\end{equation}
where $\mathbb{P}_{d}$ denotes the set of polynomials of degree less
than or equal to $d$.
%We have $S_P^d\subset H^s(\Omega)$ for $s<\frac32$.
%We will also employ a slight variation of $S_P^1$, that is the space $\tilde{S}_P^1$ of piecewise affine functions that vanish on the boundary of $\Omega$.
Then the Taylor-Hood finite element spaces are $V_P=V\cap (S_P^d)^n$ and $Q_P = Q \cap C(\Omega) \cap S_P^{d-1}$,
and the {\em Galerkin approximation} $(u_P,p_P) \in V_P \times Q_P$ of $(u,p)$ from $V_P \times Q_P$ is characterized by
\begin{equation}\label{e:Stokes-discrete}
a(u_P,v) - b(v,p_P) - b(u_P,q) = \langle f,v\rangle_{L^2},
\qquad \textrm{for all}\quad
(v,q) \in V_P \times Q_P.
\end{equation}
It is proved in \cite{Boff94,Boff97} that for any $d\geq2$ and for $n\in\{2,3\}$,
the pair $(V_P,Q_P)$ satisfies the stability property 
\begin{equation}\label{e:Brezzi-stability}
\|q\|_{Q} \leq C_s \sup_{v\in V_P} \frac{b(v,q)}{\|v\|_V},
\qquad 
q\in Q_P,
\end{equation}
with $C_s$ depending only on the initial partition $P_0$,
under the sufficient condition that $P_0$ contains at least three simplices and each simplex has at least one vertex in $\Omega$.
Earlier work on the stability of the Taylor-Hood elements include \cite{BercPiro79,Verf84,Sten87,BrezFort91,BrezFalk91}.

Throughout this paper, we are going to assume that the stability \eqref{e:Brezzi-stability} holds with $C_s$ depending only on $P_0$.
This assumption implies the well-posedness of the discrete problem \eqref{e:Stokes-discrete}, as well as the {\em a priori} error estimate
\begin{equation}\label{e:a-priori-error}
\|(u-u_P,p-p_P)\|_{V\times Q} \leq C_s' \inf_{(v,q)\in V_P\times Q_P} \|(u-v,p-q)\|_{V\times Q},
\end{equation}
where the constant $C_s'$ depends only on $C_s$ and geometry of the domain $\Omega$.
Moreover, for any $(w,r)\in V_P\times Q_P$ we have
\begin{equation}\label{e:Babuska-stability}
\|(w,r)\|_{V\times Q} \leq C_s' \sup_{(v,q)\in V_P\times Q_P} \frac{a(w,v)-b(w,q)-b(v,r)}{\|(v,q)\|_{V\times Q}}.
\end{equation}

We close this section by remarking that the Galerkin problem \eqref{e:Stokes-discrete} can be given a convenient operator formulation.
Let $X_P=V_P\times Q_P$, and let $j_P:X_P\to X$ be the natural injection.
Then \eqref{e:Stokes-discrete} is simply
\begin{equation}\label{e:Stokes-discrete-op}
j_P'Aj_P(u_P,p_P) = j_P'F ,
\end{equation}
where $j_P':X'\to X_P$ is the dual of $j_P$,
and $A$ and $F$ are defined earlier, namely in \eqref{e:Stokes-operator} and in \eqref{e:Stokes-rhs-op}, respectively.
The significance of the stability assumption in this formulation is that \eqref{e:Babuska-stability} gives not only the invertibility of 
the operator $A_P=j_P'Aj_P:X_P\to X_P$, but also it implies the bound $\|A_P^{-1}\|\leq C_s'$.

%%%%%%%%%%%%%%%%%%%%%%%%%%%%%%%%%%%%%%%%%%%%%%%%%%%%%%%%%%%%%%%%%%%%%%%%%%%%%%
\section{{\em A posteriori} error estimators}
\label{s:error-estimators}

For $P\in\conf(P_0)$, let us denote by $E_P$ the set of interior edges (or faces if $n=3$) in the partition $P$.
We also let $h_\tau = |\tau|^{1/n}$ and $h_e = |e|^{1/(n-1)}$,
where $|\tau|$ and $|e|$ are the $n$ and $(n-1)$ dimensional volumes, respectively, of $\tau\in P$ and $e\in E_P$.
Moreover, for $Q\subset P$, let $E_Q$ denote the set of interior edges of $Q$, i.e., 
the set of edges that are adjacent to two triangles from $Q$.
Then for $Q\subset P$, we introduce the residual based {\em a posteriori} error estimators
\begin{equation}\label{e:eta-0}
\eta_0(P,Q)
= \sum_{\tau\in Q} h_\tau^2 \|f+\Delta u_P - \nabla p_P\|_{L^2(\tau)}^2
+ \sum_{e\in E_Q} h_e \| [ \partial_\nu u_P ] \|_{L^2(e)}^2 ,
\end{equation}
\begin{equation}\label{e:eta-1}
\eta_1(P,Q)
= \eta_0(P,Q)
+ \sum_{\tau\in Q} \|\nabla\cdot u_P\|_{L^2(\tau)}^2,
\end{equation}
and
\begin{equation}\label{e:eta-2}
\eta_2(P,Q)
= \eta_0(P,Q)
+ \sum_{\tau\in Q} h_\tau \|\nabla\cdot u_P |_{\tau} \|_{L^2(\partial\tau)}^2 ,
\end{equation}
where $[\partial_\nu u_P]$ is the jump in the normal derivative of $u_P$ across the edge $e$.
It is understood that the differential operators $\Delta$ and $\partial_\nu$ act on vector functions such as $u_P$ component-wise.
The estimator $\eta_1$ was introduced in \cite{Verf89}, and the estimator $\eta_0$ has been proposed in \cite{MSV08}, as a variation on $\eta_1$.
The estimator $\eta_2$ seems to be new.

As shown in \cite{Verf89}, we have the equivalence
\begin{equation}\label{e:eta-1-equivalence}
\|u-u_P\|_V^2 + \|p-p_P\|_Q^2 
\lesssim
\eta(P,P)
\lesssim \|u-u_P\|_V^2 + \|p-p_P\|_Q^2 + \osc(P),
\end{equation}
for the error estimator $\eta = \eta_1$,
where the {\em oscillation} is defined as
\begin{equation}\label{e:oscillation}
\osc(P) = \min_{g\in (S^{d-2}_P)^n} \sum_{\tau\in P} h_\tau^2 \|f-g\|_{L^2(\tau)}^2.
\end{equation}
Here and in what follows, we shall often dispense with giving explicit
names to constants, and use the Vinogradov-style notation $X\lesssim
Y$, which means $X\leq C\cdot Y$ with some constant $C$ that is allowed to
depend only on $P_0$ and (the geometry of) the domain $\Omega$.
%We also use the notation $X\eqsim Y$ to mean both $X\lesssim Y$ and $Y\lesssim X$.
Moreover, even when we give names to constants, we will not explicitly mention that
the constants may depend on $P_0$ and $\Omega$, and this dependence will always be implicitly assumed.

The equivalence \eqref{e:eta-1-equivalence} also holds for $\eta=\eta_0$, because we have 
\begin{equation}\label{e:div-u-bound-edge-res}
\|\nabla\cdot u_P\|_{L^2(\Omega)}^2
\lesssim 
\sum_{e\in E_P} h_e \| [ \partial_\nu u_P ] \|_{L^2(e)}^2 ,
\end{equation}
and hence
\begin{equation}\label{e:global-equivalence-01}
\eta_0(P,P) \leq \eta_1(P,P) \lesssim \eta_0(P,P) ,
\end{equation}
see \cite[\S3.3]{MSV08} and \cite*[Proposition 5.4]{BMN02}.
Now that we have \eqref{e:eta-1-equivalence} for both $\eta_0$ and $\eta_1$,
we get it also for $\eta_2$, because
\begin{equation}\label{e:global-equivalence-021}
\eta_0(P,P) \leq \eta_2(P,P) \lesssim \eta_1(P,P) ,
\end{equation}
where the second inequality follows from
$\|\nabla\cdot u_P|_\tau\|_{L^2(\partial\tau)}\lesssim h_\tau^{-1/2}\|\nabla\cdot u_P\|_{L^2(\tau)}$.
To reiterate, the global equivalence \eqref{e:eta-1-equivalence} holds for all three estimators $\eta_0$, $\eta_1$, and $\eta_2$.

A convenient fact is that each estimator dominates the oscillation, in the sense that 
\begin{equation}\label{e:estimator-dominates-oscillation}
\osc(P) \leq \eta_0(P,P) , 
\qquad\textrm{for}\quad
P\in\conf(P_0).
\end{equation}
This is immediate because in \eqref{e:eta-0} we have $\Delta u_P - \nabla p_P\in (S^{d-2}_P)^n$,
while \eqref{e:oscillation} involves a minimization over the space $(S^{d-2}_P)^n$.

%However, we keep this term as we were not able to localize the arguments from \cite{BMN02,MSV08}.

By standard arguments, we easily get local discrete upper bounds for $\eta_1$ and $\eta_2$, which we record in the next lemma.
In the statement of the lemma, note that $P\setminus P'$ is the set of triangles in $P$ that are refined as one goes from $P$ to $P'$.

\begin{lemma}\label{l:local-upper-bound}
For $P,P'\in\conf(P_0)$ with $P\preceq P'$, we have the local discrete upper bound
\begin{equation}\label{e:local-upper-bound}
\|u_{P'}-u_P\|_V^2 + \|p_{P'}-p_P\|_Q^2 \lesssim \eta_1(P,P\setminus P').
\end{equation}
Moreover, we have the local equivalence
\begin{equation}\label{e:local-equivalence}
\alpha \eta_1(P,Q) \leq \eta_2(P,Q) \leq \beta \eta_1(P,Q),
\qquad
Q\subset P,
\end{equation}
where $\alpha>0$ and $\beta$ are constants.
\end{lemma}

\begin{proof}
For any $(v,q)\in(V_{P'},Q_{P'})$ and $(v_P,q_P)\in(V_{P},Q_{P})$, 
the equation \eqref{e:Stokes-discrete} and integration by parts give
\begin{multline}
a(u_{P'}-u_P,v) - b(v,p_{P'}-p_P) - b(u_{P'}-u_P,q) \\
=
a(u_{P'}-u_P,v-v_P) - b(v-v_P,p_{P'}-p_P) - b(u_{P'}-u_P,q-q_P) \\
=
\langle f,v-v_P\rangle - a(u_P,v-v_P) + b(v-v_P,p_P) + b(u_P,q-q_P) \\
=
\langle f,v-v_P\rangle
+ \sum_{\tau\in P} \big( \int_\tau \Delta u \cdot (v-v_P) - \int_{\partial\tau} \partial_\nu u_P \cdot (v-v_P) \big) \\
- \int_{\Omega} (v-v_P) \cdot \nabla p_P
+ \int_{\Omega} (q-q_P) \nabla\cdot u_P .
\end{multline}
Let $\omega = \intr\bigcup_{\tau\in P\setminus P'}\overline\tau$,
i.e., let $\omega$ be the interior of the region covered by the refined triangles.
Then we set $(v_P,q_P)$ to be equal to $(v,q)$ in $\Omega\setminus\omega$,
and equal to the Scott-Zhang interpolator of $(v,q)$ in $\omega$.
In doing so, we choose the Scott-Zhang interpolator to be adapted to the boundary of $\omega$,
thus ensuring that $v_P\in V_P$ and $q_P\in Q_P$, cf. \cite{SZ90}.
With this preparation, we can continue the chain of reasoning as follows.
\begin{multline}
a(u_{P'}-u_P,v) - b(v,p_{P'}-p_P) - b(u_{P'}-u_P,q) \\
=
\sum_{\tau\in P} \int_\tau (f+\Delta u - \nabla p_P ) \cdot (v-v_P) 
- \sum_{e\in E_P} \int_{e} [\partial_\nu u_P] \cdot (v-v_P)
+ \int_{\Omega} (q-q_P) \nabla \cdot u_P \\
\leq 
\sum_{\tau\in P\setminus P'} \|f+\Delta u - \nabla p_P\|_{L^2(\tau)} \|v-v_P\|_{L^2(\tau)} \\
+ \sum_{e\in E_P\setminus E_{P'}} \| [\partial_\nu u_P] \|_{L^2(\tau)} \|v-v_P\|_{L^2(\tau)}
+ \sum_{\tau\in P\setminus P'} \|\nabla \cdot u_P\|_{L^2(\tau)} \|q-q_P\|_{L^2(\tau)} .
\end{multline}
Since $(u_{P'}-u_P,p_{P'}-p_P)\in(V_{P'},Q_{P'})$,
we use the stability \eqref{e:Babuska-stability}, 
in combination with standard estimates for the Scott-Zhang interpolator and local finiteness,
to establish \eqref{e:local-upper-bound}.

The second inequality in \eqref{e:local-equivalence}, namely $\eta_2(P,Q)\lesssim\eta_1(P,Q)$, follows from 
the inverse estimate $\|\nabla\cdot u_P|_\tau\|_{L^2(\partial\tau)}\lesssim h_\tau^{-1/2}\|\nabla\cdot u_P\|_{L^2(\tau)}$.
In order to prove the other inequality $\eta_1(P,Q)\lesssim\eta_2(P,Q)$, we localize the argument from \cite{BMN02,MSV08}.
By putting $v=0$ in \eqref{e:Stokes-discrete}, we see that $\nabla\cdot u_P$ is $L^2$-orthogonal to the pressure space $Q_P$.
Since $\int_\Omega\nabla\cdot u_P=0$ by the divergence theorem, this means that $\nabla\cdot u_P$ is $L^2$-orthogonal to the full space $C(\Omega)\cap S_P^{d-1}$.
In particular, 
$(\nabla\cdot u_P)|_{\omega}$ is $L^2$-orthogonal to $H^1_0(\omega)\cap S_P^{d-1}$ on $\omega$, 
where $\omega = \intr\bigcup_{\tau\in Q}\overline\tau$,
meaning that
\begin{equation}
\|\nabla\cdot u_P\|_{L^2(\omega)} \leq \|\nabla\cdot u_P - q\|_{L^2(\omega)} ,
\end{equation}
for any $q\in H^1_0(\omega)\cap S_P^{d-1}$.
Let $S^* = \{g\in L^2(\omega)\cap S_P^{d-1} : g \perp_{L^2(\omega)} H^1_0(\omega)\cap S_P^{d-1}\}$.
Then we claim that
\begin{equation}
\|g\|_{L^2(\omega)}^2
\lesssim 
\sum_{e\in E_Q} h_e \| [g] \|_{L^2(e)}^2 
+
\sum_{\{e\in E_P:e\subset\partial\omega\}} h_e \| g \|_{L^2(e)}^2 ,
\end{equation}
for any $g\in S^*$, where $[g]$ denotes the jump of $g$ across $e$.
Indeed, the right hand side defines a (squared) norm on $S^*$,
since the vanishing of this quantity implies that $g\in C(\omega)\cap S^{d-1}_P$ and $g|_{\partial\omega}=0$,
which means that $g=0$ by $g \perp_{L^2(\omega)} H^1_0(\omega)\cap S_P^{d-1}$.
The local scaling by $h_e$ can be deduced by a local homogeneity argument.
Finally, plugging in $g=\nabla\cdot u_P$ and using straightforward bounds, 
the proof is completed.
\end{proof}

We end this section with the following standard result.

\begin{lemma}\label{l:est-red}
\begin{enumerate}[a)]
\item
Let $P,P'\in\conf(P_0)$ be such that $P\preceq P'$,
and let
\begin{equation}\label{e:Dorfler-1}
\eta_2(P,P\setminus P') \geq \theta \eta_2(P,P) ,
\end{equation}
for some $0<\theta\leq1$.
Then we have
\begin{equation}\label{e:est-red}
\eta_2(P',P')
\leq 
\mu\eta_2(P,P)
+ \gamma \| ( u_P - u_{P'}, p_P - p_{P'} ) \|_{V\times Q}^2 ,
\end{equation}
with $\mu<1$ and $\gamma$ depending only on $\theta$.
\item
Let $P,P'\in\conf(P_0)$ be such that $P\preceq P'$,
and for some $0<\mu<\frac12$ suppose that 
\begin{equation}\label{e:est-red-1}
\eta_2(P',P') \leq \mu \eta_2(P,P) .
\end{equation}
Then we have
\begin{equation}
\eta_2(P,P\setminus P') \geq \theta^*(1-2\mu) \eta_2(P,P) ,
\end{equation}
where $\theta^*>0$ is a constant independent of $\mu$.
\end{enumerate}
\end{lemma}

\begin{proof}
a)
By standard arguments, one can prove that
\begin{equation}
\eta_2(P',P')
\leq 
(1+\delta)\eta_2(P,P)
- \lambda(1+\delta) \eta_2(P,P\setminus P')
+ C_\delta \| ( u_P - u_{P'}, p_P - p_{P'} ) \|_{V\times Q}^2,
\end{equation}
for any $P,P'\in\conf(P_0)$ with $P\preceq P'$,
and for any $\delta>0$, 
where $C_\delta$ may depend on $\delta$,
and $\lambda>0$ is independent of $\delta$.
Upon using \eqref{e:Dorfler-1}, this gives
\begin{equation}
\eta_2(P',P')
\leq 
(1+\delta) (1-\lambda\theta) \eta_2(P,P)
+ C_\delta \| ( u_P - u_{P'}, p_P - p_{P'} ) \|_{V\times Q}^2,
\end{equation}
and choosing $\delta>0$ small enough we get \eqref{e:est-red}.

b)
Observe that
\begin{multline}
\eta_2(P,P\cap P')
\leq 
2\eta_2(P',P\cap P')
+
\sum_{\tau\in P\cap P'} 2 h_\tau^2 \|\Delta (u_P-u_{P'}) - \nabla (p_P-p_{P'})\|_{L^2(\tau)}^2 \\
+ \sum_{e\in E_P\cap E_{P'}} 2 h_e \| [ \partial_\nu (u_P-u_{P'}) ] \|_{L^2(e)}^2
+ \sum_{\tau\in P\cap P'} 2 h_\tau \|\nabla\cdot (u_P-u_{P'}) |_{\tau} \|_{L^2(\partial\tau)}^2 \\
\leq
2\eta_2(P',P\cap P')
+ C \| ( u_P - u_{P'}, p_P - p_{P'} ) \|_{V\times Q}^2.
\end{multline}
Using this observation, the property \eqref{e:est-red-1}, and the local upper bound from Lemma \ref{l:local-upper-bound}, we infer
\begin{equation}
\begin{split}
(1-2\mu)\eta_2(P,P)
&\leq 
\eta_2(P,P\setminus P') + \eta_2(P,P\cap P') - 2 \eta_2(P',P') \\
&\leq 
\eta_2(P,P\setminus P') %+ 2 \eta_2(P',P\cap P') - 2 \eta_2(P',P') 
+ C \| ( u_P - u_{P'}, p_P - p_{P'} ) \|_{V\times Q}^2 \\
&\lesssim 
\eta_2(P,P\setminus P'),
\end{split}
\end{equation}
which concludes the proof.
\end{proof}

%%%%%%%%%%%%%%%%%%%%%%%%%%%%%%%%%%%%%%%%%%%%%%%%%%%%%%%%%%%%%%%%%%%%%%%%%%%%%%
\section{Convergence rates}
\label{s:convergence}

We are ready to start our discussion of adaptive algorithms and their convergence rates.
A template of an adaptive finite element method is displayed as Algorithm \ref{a:afem}.
The {\em a posteriori} error estimator $\eta$ can be chosen to be $\eta_0$ from \eqref{e:eta-0},
$\eta_1$ from \eqref{e:eta-1}, or $\eta_2$ from \eqref{e:eta-2}.
For theoretical purposes, we think of the algorithm as generating an infinite sequence of triples
$\{(P_k,u_k,p_k)\}$, where $P_k\in\conf(P_0)$ and $(u_k,p_k)\in V_{P_k}\times Q_{P_k}$ for all $k\in\N_0$.

\vspace{2mm}
\begin{algorithm}[H]
\SetKwInOut{Params}{input}
\SetKwInOut{Output}{output}
\SetKwFor{For}{for}{do}{endfor}
\Params{conforming partition $P_0$, and $0<\theta\leq1$}
\Output{$P_k\in\conf(P_0)$ and $(u_k,p_k)\in V_{P_k}\times Q_{P_k}$ for all $k\in\N_0$}
\BlankLine
\For{$k=0,1,\ldots$}{
Compute $(u_k,p_k)\in V_{P_k}\times Q_{P_k}$ as the Galerkin approximation of $(u,p)$\;\nllabel{a:galsolve-pos}
Identify a minimal (up to a constant factor) set $R_k\subset P_k$ of triangles satisfying
\begin{equation}\label{e:Dorlfer-property}
\eta(P_k,R_k) \geq \theta \eta(P_k,P_k);\;
\end{equation}
Set $P_{k+1}=\refine(P_k,R_k)$\;\nllabel{a:make-conf-pos}
}
\caption{Adaptive FEM}\label{a:afem}
\end{algorithm}
\vspace{2mm}

The first important question is if and how fast the approximations $(u_k,p_k)$ converge to the exact solution $(u,p)$ as $k\to\infty$.
The following plain convergence result was obtained in \cite[\S3.3]{MSV08} for $\eta=\eta_1$
and in \cite[\S4.4]{Sieb11} for $\eta=\eta_0$.
Following \cite*{AFLP10}, we give here a slightly different proof.

\begin{lemma}\label{l:conv-plain}
In the context of Algorithm \ref{a:afem},
let $\eta$ be one of $\eta_0$, $\eta_1$ and $\eta_2$.
Then we have $(u_k,p_k)\to(u,p)$ in $V\times Q$ as $k\to\infty$.
\end{lemma}

\begin{proof}
By \cite[Lemma 4.2]{MSV08},
we have $(u_k,p_k)\to(u_\infty,p_\infty)$ in $V\times Q$ as $k\to\infty$,
for some $(u_\infty,p_\infty)\in V\times Q$.

On the other hand, since $\eta_0(P_k,R_k)\leq\eta_2(P_k,R_k)$ and $\eta_2(P_k,P_k)\lesssim\eta_0(P_k,P_k)$ from the global equivalences \eqref{e:global-equivalence-01} and \eqref{e:global-equivalence-021}, 
the D\"orfler property \eqref{e:Dorlfer-property} for $\eta=\eta_0$ implies the same for $\eta=\eta_2$, with possibly a different constant $\theta>0$.
Similarly, by the equivalence \eqref{e:local-equivalence}, the D\"orfler property \eqref{e:Dorlfer-property} for $\eta=\eta_1$ implies the same for $\eta=\eta_2$, 
with possibly a different constant $\theta>0$.
The latter argument runs also in the other direction, since \eqref{e:local-equivalence} is a local equivalence.
To conclude, we can assume the D\"orfler property \eqref{e:Dorlfer-property} for both $\eta=\eta_1$ and $\eta=\eta_2$, 
with possibly different constants $\theta>0$.

In any case, by Lemma \ref{l:est-red} a), there exist constants $\mu<1$ and $\gamma\geq0$ such that
\begin{equation}
\eta_2(P_{k+1},P_{k+1})
\leq 
\mu\eta_2(P_{k},P_{k})
+ \gamma \| ( u_k - u_{k+1}, p_k - p_{k+1} ) \|_{V\times Q}^2 ,
\end{equation}
for all $k\in\N$.
The last term converges to $0$ as $k\to\infty$, since $(u_k,p_k)$ is convergent.
Thus introducing the abbreviation $e_k=\eta_2(P_{k},P_{k})$, we have
\begin{equation}
e_{k+1} \leq \mu e_k + \alpha_k ,
\end{equation}
with $\alpha_k\to0$.
Let $\eps>0$, and let $k$ be such that $\alpha_{k+m}\leq\eps$ for all $m\geq0$.
Then we have
\begin{equation}
e_{k+m} \leq \mu^m e_k + \eps (1+\mu+\ldots+\mu^{m-1}) \leq \mu^m e_k + \frac\eps{1-\mu} ,
\end{equation}
for all $m\geq0$.
This shows that $\limsup_{k\to\infty}e_k\leq\eps/(1-\mu)$.
Since $\eps>0$ is arbitrary, and $e_k\geq0$, we conclude that $\lim_{k\to\infty}e_k=0$.
Finally, the global upper bound in \eqref{e:eta-1-equivalence} implies the convergence
$(u_k,p_k)\to(u,p)$ in $V\times Q$ as $k\to\infty$.
\end{proof}

We have the following {\em conditional} error reduction theorem for all three estimators.

\begin{theorem}\label{t:conv}
In the context of Algorithm \ref{a:afem},
let $\eta$ be one of $\eta_0$, $\eta_1$ and $\eta_2$.
Assume that there exists a constant $c>0$ such that
\begin{equation}\label{e:general-quasi-orthogonality}
\sum_{k=\ell}^{N} \| ( u_k - u_{k+1}, p_k - p_{k+1} ) \|_{V\times Q}^2
\leq 
c \| ( u - u_{\ell}, p - p_{\ell} ) \|_{V\times Q}^2 ,
\end{equation}
for any integers $\ell$ and $N$.
Then there are constants $\rho<1$ and $C>0$ such that
\begin{equation}\label{e:eta-decay}
\eta(P_k,P_k) \leq C \rho^{k-\ell} \eta(P_\ell,P_\ell) ,
\end{equation}
for all $k\geq\ell\geq0$.
In particular, we have
\begin{equation}
\|(u-u_k,p-p_k)\|_{V\times Q} \leq C' \rho^{k},
\end{equation}
for some constant $C'$.
\end{theorem}

\begin{proof}
As we have discussed in the proof of Lemma \ref{l:conv-plain},
we can assume the D\"orfler property \eqref{e:Dorlfer-property} for both $\eta=\eta_1$ and $\eta=\eta_2$, 
with possibly different constants $\theta>0$.
Hence by Lemma \ref{l:est-red} a), there exist constants $\mu<1$ and $\gamma\geq0$ such that
\begin{equation}
\eta_2(P_{k+1},P_{k+1})
\leq 
\mu\eta_2(P_{k},P_{k})
+ \gamma \| ( u_k - u_{k+1}, p_k - p_{k+1} ) \|_{V\times Q}^2 ,
\end{equation}
for all $k\in\N$.
Then invoking the assumption \eqref{e:general-quasi-orthogonality},
and the global upper bound from \eqref{e:eta-1-equivalence},
we get
\begin{equation}\label{e:summ-eta}
\begin{split}
\sum_{k=\ell}^{N} \eta_2(P_{k+1},P_{k+1})
&\leq 
\mu \sum_{k=\ell}^{N} \eta_2(P_{k},P_{k})
+ \gamma \sum_{k=\ell}^{N} \| ( u_k - u_{k+1}, p_k - p_{k+1} ) \|_{V\times Q}^2 \\
&\leq 
\mu \sum_{k=\ell}^{N} \eta_2(P_{k},P_{k})
+ \gamma c \| ( u - u_{\ell}, p - p_{\ell} ) \|_{V\times Q}^2 \\
&\leq 
\mu \sum_{k=\ell}^{N} \eta_2(P_{k},P_{k})
+ C \eta_2(P_{\ell},P_{\ell}) ,
\end{split}
\end{equation}
for any $\ell\in\N$ and $N\geq\ell$.
Since $\mu<1$, this implies the convergence of the series $\sum_k\eta_2(P_k,P_k)$, and thus
\begin{equation}\label{e:alpha-eta}
\alpha_{\ell} \leq \mu \alpha_\ell + (1+C) \eta_2(P_{\ell},P_{\ell}) ,
\qquad\textrm{or}\qquad
\alpha_{\ell} \leq \frac{1+C}{1-\mu} \eta_2(P_{\ell},P_{\ell}) ,
\end{equation}
where we have added $\eta_2(P_{\ell},P_{\ell})$ to the both sides of \eqref{e:summ-eta}, and introduced
\begin{equation}
\alpha_{\ell} = \sum_{k=\ell}^{\infty} \eta_2(P_{k},P_{k}) .
\end{equation}
As a consequence, we get
\begin{equation}
\alpha_{\ell+1} = \alpha_\ell - \eta_2(P_{\ell},P_{\ell}) \leq \frac{\mu+C}{1+C} \alpha_\ell ,
\end{equation}
which means that the quantity $\alpha_\ell$ decays geometrically.
This yields
\begin{equation}
\eta_2(P_k,P_k)\leq\alpha_k \leq \big( \frac{\mu+C}{1+C} \big)^{k-\ell} \frac{1+C}{1-\mu} \eta_2(P_{\ell},P_{\ell}) ,
\end{equation}
where we have taken into account \eqref{e:alpha-eta}.
Finally, the same geometric decay for both $\eta_0$ and $\eta_1$ follow from the equivalences \eqref{e:global-equivalence-01} and \eqref{e:global-equivalence-021}.
\end{proof}

Now we address the question of convergence rate.
We start by defining 
\begin{equation}
E_P(u,p) = \inf_{(v,q)\in V_P\times Q_P} \| (u-v,p-q) \|_{V\times Q} ,
\end{equation}
for $P\in\conf(P_0)$, and
\begin{equation}
\sigma^*_N(u,p) = \inf_{P\in\tstP_N} \big( E_P(u,p)^2 + \osc(P) \big)^{\frac12},
\end{equation}
for $N\in\N$, where $\tstP_N=\{P\in\conf(P_0):\#P-\#P_0\leq N\}$.
Note that the oscillation depends on $u$ and $p$ implicitly through the equation \eqref{e:Stokes-weak}.
Following \cite{CKNS08}, we then define the {\em modified approximation class}
\begin{equation}\label{e:approx-class-mod}
\tstA_*^s = \{(u,p)\in V\times Q : |(u,p)|_{\tstA_*^s} \equiv \sup_{N\in\N} N^s\sigma^*_N(u,p)<\infty\} ,
\end{equation}
for $s>0$.
Thus, $(u,p)\in\tstA_*^s$ if and only if for each $N\in\N$, there exists a partition $P\in\conf(P_0)$ with $\#P-\#P_0\leq N$, 
such that 
\begin{equation}
E_P(u,p)^2+\osc(P)
\leq c (\#P-\#P_0)^{-2s} ,
\end{equation}
where the constant $c=c(u,p)$ is independent of $P$.
The greatest lower bound for such constants $c$ is the quantity $|(u,p)|_{\tstA_*^s}^2$.

The following is one of our main results alluded to in the Introduction.

\begin{theorem}\label{t:rate}
In the context of Algorithm \ref{a:afem},
let $\eta$ be either $\eta_1$ or $\eta_2$, and let $\theta>0$ be small enough.
A sufficient condition is $\theta<\theta^*$ for $\eta=\eta_2$, and $\theta<\frac\alpha\beta\theta^*$ for $\eta=\eta_1$.
Suppose that $f\in L^2(\Omega,\R^n)$ and that $(u,p)\in\tstA_*^s$ for some $s>0$.
In addition, assume \eqref{e:general-quasi-orthogonality} for the solution sequence $\{(u_k,p_k)\}$.
Then there exists a constant $c>0$ such that
\begin{equation}\label{e:optimal-rate}
\| (u-u_k,p-p_k) \|_{V\times Q}^2+\osc(P_k)
\leq
c |(u,p)|_{\tstA_*^s}^2 (\#P_k-\#P_0)^{-2s} .
\end{equation}
\end{theorem}

\begin{proof}
We will prove the theorem only for $\eta=\eta_1$, since this case is slightly non-standard. 
Our strategy is to use the local equivalence \eqref{e:local-equivalence} to relate $\eta_1$ with $\eta_2$,
and use standard arguments for $\eta_2$.
%To avoid confusion, note that the estimator $\eta_2$ will be used as a theoretical tool.
For $P\in\conf(P_0)$, let 
\begin{equation}
e(P) = \| (u-u_P,p-p_P) \|_{V\times Q}^2+\osc(P).
\end{equation}
Note that from the {\em a priori} estimate \eqref{e:a-priori-error}, and the definition of oscillation \eqref{e:oscillation},
we have the weak monotonicity
\begin{equation}
e(P') \lesssim e(P) ,
\end{equation}
for any refinement $P'\in\conf(P_0)$ of $P$.

By definition of $\tstA_*^s$, there exists a partition $P\in\conf(P_0)$ such that 
\begin{equation}
\#P - \#P_0 \leq \eps_k^{-1/s} |(u,p)|_{\tstA_*^s}^{1/s},
\qquad\textrm{and}\qquad
E_P(u,p)^2 + \osc(P) \leq \eps_k^2 ,
\end{equation}
where $\eps_k= \delta\eta_2(P_k,P_k)$, with $\delta>0$ a small constant.
Let $P' = P\oplus P_k$.
Then the global lower bound \eqref{e:eta-1-equivalence} and the {\em a priori} estimate \eqref{e:a-priori-error} yield
\begin{equation}
\eta_2(P',P') \lesssim e(P') \lesssim e(P) \lesssim \eps_k^2 = \delta\eta_2(P_k,P_k).
\end{equation}
Upon choosing $\delta>0$ small enough, this implies
$\eta_2(P',P') \leq \mu\eta_2(P_k,P_k)$ with $\mu=\frac12(1-\frac{\beta\theta}{\alpha\theta^*})$,
and by Lemma \ref{l:est-red}, we have $\eta_2(P_k,P_k\setminus P')\geq\frac{\beta\theta}{\alpha}\eta_2(P_k,P_k)$.
Under the local equivalence \eqref{e:local-equivalence}, it becomes $\eta_1(P_k,P_k\setminus P')\geq\theta\eta_1(P_k,P_k)$.
Since by construction, $R_k\subset P_k$ is a minimal (up to a constant factor) set satisfying $\eta_1(P_k,R_k) \geq \theta \eta_1(P_k,P_k)$,
we infer $\#R_k\lesssim \#(P_k\setminus P')$, and hence
\begin{equation}
\#R_k\lesssim \#P_k - \#P' \leq \#P - \#P_0 \leq \eps_k^{-1/s} |(u,p)|_{\tstA_*^s}^{1/s} ,
\end{equation}
where we have used the property \eqref{e:overlay} of overlays.
Now we invoke \eqref{e:complete} and the geometric decay \eqref{e:eta-decay}, to conclude
\begin{equation}
\#P_\ell-\#P_0
\lesssim \sum_{k=0}^{\ell-1} \#R_k
\lesssim |(u,p)|_{\tstA_*^s}^{1/s} \sum_{k=0}^{\ell-1} \eps_k^{-1/s}  
\lesssim \eps_\ell^{-1/s}   |(u,p)|_{\tstA_*^s}^{1/s} .
\end{equation}
By recalling that the estimators dominate oscillation, the proof is completed.
\end{proof}

%\begin{remark}
%Since the constant $c>0$ from Theorem \ref{t:conv} may depend on $u$ and $p$, we cannot rule out the possibility that the constant $c>0$ that appears in \eqref{e:optimal-rate} has a dependence on $u$ and $p$. Clarifying the nature of this constant demands future investigation.
%\end{remark}

%\begin{remark}
%For the estimator $\eta_0$, which was first introduced in \cite{MSV08}, we have established geometric error reduction, but we were unable to prove optimal convergence rates, because of the apparently nonlocal nature of the estimator.
%\end{remark}

%%%%%%%%%%%%%%%%%%%%%%%%%%%%%%%%%%%%%%%%%%%%%%%%%%%%%%%%%%%%%%%%%%%%%%%%%%%%%%
\section{Approximation classes}
\label{s:approximation}

In the preceding section, we have established optimal convergence rates with respect to the modified approximation classes \eqref{e:approx-class-mod}.
Ideally, however, one would like to have optimality with respect to the {\em standard approximation classes}
\begin{equation}\label{e:approx-class}
\tstA^s = \{(u,p)\in V\times Q : |(u,p)|_{\tstA^s} \equiv \sup_{N\in\N} N^s \inf_{P\in\tstP_N} E_P(u,p) < \infty\} ,
\end{equation}
where we recall $\tstP_N=\{P\in\conf(P_0):\#P-\#P_0\leq N\}$.
Perhaps more practical goal is to know interrelations between $\tstA^s$ and $\tstA_*^s$.
In this section, we will study those interrelations in terms of Besov space memberships of the data $f$.

It will be convenient to define the {\em oscillation classes}
\begin{equation}\label{e:osc-class}
\tstO^s = \{f\in L^2(\Omega,\R^n) : |f|_{\tstO^s} \equiv \sup_{N\in\N} N^s \inf_{P\in\tstP_N} \osc(P)^{\frac12}<\infty\} .
\end{equation}
%where we recall that $\Omega\subset\R^n$ is a polyhedral domain.
Obviously, we have $\tstA_*^s\subset\tstA^s$ for all $s>0$.
In the converse direction, we have the following well-known result.

\begin{lemma}\label{l:direct}
Let $(u,p)\in\tstA^s$ and $f\in\tstO^s$ with $s>0$, where $u$, $p$, and $f$ satisfy the equation \eqref{e:Stokes-weak}.
Then we have $(u,p)\in\tstA_*^s$ with $|(u,p)|_{\tstA_*^s}\lesssim |(u,p)|_{\tstA^s}+|f|_{\tstO^s}$.
\end{lemma}

\begin{proof}
Let $N\in\N$ be an arbitrary number.
By definition of $\tstA^s$, there exists a partition $P'\in\conf(P_0)$ such that
\begin{equation}
E_{P'}(u,p) \leq 2 N^{-s} |(u,p)|_{\tstA^s},
\qquad\textrm{and}\qquad
\#P' - \#P_0 \leq N.
\end{equation}
Similarly, by definition of $\tstO^s$, there exists a partition $P''\in\conf(P_0)$ such that
\begin{equation}
\osc(P'') \leq 2N^{-2s} |f|_{\tstO^s}^2,
\qquad\textrm{and}\qquad
\#P'' - \#P_0 \leq N.
\end{equation}
Then for $P=P'\oplus P''$ we have $\#P - \#P_0 \leq 2N$ by \eqref{e:overlay}.
Moreover, monotonicity arguments guarantee that 
\begin{equation}
E_P(u,p)^2 + \osc(P) 
\leq
E_{P'}(u,p)^2 + \osc(P'') 
\lesssim
N^{-2s} ( |(u,p)|_{\tstA^s}^2 + |f|_{\tstO^s}^2 ) ,
\end{equation}
which completes the proof.
\end{proof}

The preceding lemma makes us wonder how regular $f$ must be in order for it to be a member of $\tstO^s$.
By using quasi-uniform partitions, one can show that $H^{\alpha}(\Omega,\R^n)\subset\tstO^{s}$ for $s=\frac{\alpha+1}n$ and $\alpha\geq0$.
For instance, if we want to recover the optimal convergence rates of the lowest order Taylor-Hood elements ($d=2$), 
then this would require $f\in H^1(\Omega,\R^n)$, which appears to be a bit excessive.
As it is natural in the current setting, we would like to investigate the question in terms of the Besov regularity of $f$. %This will lead us to non-uniform partitions
Let us make precise what we mean by Besov spaces.
For $0<q\leq\infty$, the $m$-th order {\em $L^q$-modulus of smoothness} is
\begin{equation}
\omega_m(u,t,\Omega)_q=\sup_{|h|\leq t}\|\Delta_h^mu\|_{L^q(\Omega_{rh})}
\end{equation}
where $\Omega_{mh}=\{x\in\Omega:[x+mh]\subset\Omega\}$,
and $\Delta_h^m$ is the $m$-th order forward difference operator defined recursively by $[\Delta_h^1u](x)=u(x+h)-u(x)$ and $\Delta_h^ku=\Delta_h^1(\Delta_h^{k-1})u$,
i.e.,
\begin{equation}
\Delta_h^mu (x) = \sum_{k=0}^m (-1)^{m+k} \binom{m}{k} u(x+kh).
\end{equation}
Then, for $0<q,r\leq\infty$ and $\alpha\geq0$, with $m>\alpha-\max\{0,\frac1q-1\}$ being an integer, 
the {\em Besov space} $B^\alpha_{q,r}(\Omega)$ consists of those $u\in L^q(\Omega)$ for which
\begin{equation}\label{e:Besov-norm}
|u|_{B^\alpha_{q,r}(\Omega)} = \| t\mapsto t^{-\alpha-1/r}\omega_m(u,t,\Omega)_q \|_{L^r((0,\infty))},
\end{equation}
is finite.
Since $\Omega$ is bounded, being in a Besov space is a statement about the size of $\omega_m(u,t,\Omega)_q$ only for small $t$.
From this it is easy to derive the useful equivalence
\begin{equation}\label{e:Besov-norm-disc}
|u|_{B^\alpha_{q,r}(\Omega)} \eqsim \left\| (\lambda^{j\alpha}\omega_m(u,\lambda^{-j},\Omega)_q)_{j\geq0} \right\|_{\ell^r},
\end{equation}
for any constant $\lambda>1$.
The mapping $\|\cdot\|_{B^\alpha_{q,r}(\Omega)}=\|\cdot\|_{L^q(\Omega)}+|\cdot|_{B^\alpha_{q,r}(\Omega)}$ defines a norm when $q,r\geq1$ and only a quasi-norm otherwise.
So long as $m>\alpha-\max\{0,\frac1q-1\}$, different choices of $m$ will result in (quasi-) norms that are equivalent to each other.
On the other hand, if we took $m<\alpha-\max\{0,\frac1q-1\}$, then the space $B^\alpha_{q,r}$ would have been trivial in the sense that $B^\alpha_{q,r}=\Pol_{m-1}$.

We have the sub-additivity property
\begin{equation}\label{e:Besov-sub-add}
\sum_{\tau\in P} |f|_{B^{\alpha}_{q,q}(\tau)}^q
\lesssim |f|_{B^{\alpha}_{q,q}(\Omega)}^q ,
\qquad
f\in B^{\alpha}_{q,q}(\Omega) ,
\end{equation}
for $P\in\conf(P_0)$ and $0<q<\infty$.
A slightly stronger form of this is also true:
Let $\{\tau_k\}$ be a finite collection of disjoint triangles, with each $\tau_k\in P_k$ for some $P_k\in\conf(P_0)$.
Let $\hat\tau_k$ denote the star around $\tau_k$, with respect to $P_k$,
i.e., let $\hat\tau_k$ be the interior of $\bigcup\{\bar\sigma\in P_k:\bar\sigma\cap\bar\tau_k\neq\varnothing\}$.
Then we have
\begin{equation}\label{e:Besov-sub-add-overlap}
\sum_{k} |f|_{B^{\alpha}_{q,q}(\hat\tau_k)}^q
\lesssim |f|_{B^{\alpha}_{q,q}(\Omega)}^q ,
\qquad
f\in B^{\alpha}_{q,q}(\Omega) .
\end{equation}

Now we describe various embedding relationships among the Besov and Sobolev spaces.
Since $\Omega$ is bounded, it is clear that $B^\alpha_{q,r}(\Omega)\hookrightarrow B^{\alpha}_{q',r}(\Omega)$
for any $\alpha\geq0$, $0<r\leq\infty$ and $\infty\geq q>q'>0$.
From the equivalence \eqref{e:Besov-norm-disc}, we have the lexicographical ordering
$B^\alpha_{q,r}(\Omega)\hookrightarrow B^{\alpha'}_{q,r'}(\Omega)$ for $\alpha>\alpha'$ with any $0<r,r'\leq\infty$,
and
$B^\alpha_{q,r}(\Omega)\hookrightarrow B^\alpha_{q,r'}(\Omega)$ for $0<r<r'\leq\infty$.
Nontrivial embeddings are
$B^\alpha_{q,r}(\Omega)\hookrightarrow B^{\alpha'}_{q',r}(\Omega)$
for $\frac{\alpha-\alpha'}n=\frac1q-\frac1{q'}>0$,
and
\begin{equation}\label{e:Besov-Lebesgue}
B^\alpha_{q,q}(\Omega)\hookrightarrow L^{r}(\Omega),
\qquad \textrm{for} \qquad
\frac\alpha{n} = \frac1q-\frac1r>0.
\end{equation}
Finally, we recall the fact that $B^\alpha_{2,2}(\Omega)=H^\alpha(\Omega)$ for all $\alpha>0$.

An important tool we will need is the {\em Whitney estimate}
\begin{equation}\label{e:Whitney-est}
\inf_{g\in\Pol_m} \|f-g\|_{L^q(G)} \lesssim \omega_{m+1}(f,\diam\,G,G)_q,
\qquad f\in L^q(G),
\end{equation}
that holds for any convex domain $G\subset\R^n$, 
with the implicit constant depending only on $n$, $m$, and $q$, see \cite{DL04a}.
The same estimate is also true when $G$ is the star around $\tau\in P$ for some partition $P\in\conf(P_0)$,
with the implicit constant additionally depending on the shape regularity constant of $\conf(P_0)$, see \cite{GM13}.

In the ensuing discussions, we will often need vector-valued versions of the Besov and other function spaces,  
that should strictly speaking be denoted by $B_{p,q}^\alpha(\Omega,\R^n)$ or $L^p(\Omega)^n$ etc.
However, for simplicity of notations, we will simply write $B_{p,q}^\alpha(\Omega)$ or $L^p(\Omega)$ etc. to mean the same things.

\begin{theorem}\label{t:direct}
We have $B^\alpha_{q,q}(\Omega,\R^n) \subset \tstO^s$ with $s=\frac{\alpha+1}n$, 
as long as $0<q<\infty$, $\frac\alpha{n}\geq\frac1q - \frac12$, and $\alpha<d-1+\max\{0,\frac1q-1\}$.
\end{theorem}

\begin{proof}
In this proof, it is understood that all Besov seminorms are defined using $\omega_{d-1}$.
Hence in particular, we have $|g|_{B^{\alpha}_{q,q}}=0$ for $g\in\Pol_{d-2}$.
For any $g\in(\Pol_{d-2})^n$ and any convex domain $G\subset\R^n$, we have 
\begin{equation}
\|f - g\|_{L^2(G)} 
\lesssim 
\|f-g\|_{L^q(G)} + |f|_{B^{\alpha}_{q,q}(G)},
\end{equation}
by the embedding \eqref{e:Besov-Lebesgue},
and
\begin{equation}
\|f-g\|_{L^q(G)} 
\lesssim \omega_{d-1}(f,\diam\,G,G)_q
\lesssim |f|_{B^{\alpha}_{q,q}(G)} ,
\end{equation}
by the Whitney estimate \eqref{e:Whitney-est}.
Recalling that $h_\tau=|\tau|^{1/n}$, a homogeneity argument gives 
\begin{equation}
\osc(P) = \min_{g\in (S^{d-2}_P)^n} \sum_{\tau\in P} h_\tau^2 \|f-g\|_{L^2(\tau)}^2
\lesssim \sum_{\tau\in P} |\tau|^{2\delta} |f|_{B^{\alpha}_{q,q}(\tau)}^2 ,
\end{equation}
for $P\in\conf(P_0)$,
where $\delta = \frac{\alpha+1}n+\frac12-\frac1q\geq\frac1n$.

The rest of the proof follows that of Proposition 5.2 in \cite{BDDP02};
we include it here for completeness.
Let
\begin{equation}\label{e:err-ind-0}
e(\tau,P) = |\tau|^{2\delta} |u|_{B^{\alpha}_{q,q}(\tau)}^2,
\end{equation}
for $\tau\in P$ and $P\in\conf(P_0)$.
Then for any given $\eps>0$,
below we will specify a procedure to generate a partition $P\in\conf(P_0)$ satisfying 
\begin{equation}\label{e:rho-bnd-pf}
\sum_{\tau\in P} e(\tau,P) \lesssim (\#P) \eps,
\end{equation}
and 
\begin{equation}\label{e:card-P-bnd-pf}
\#P - \#P_0 \leq c \eps^{-1/(1+2s)} |f|_{B^{\alpha}_{q,q}(\Omega)}^{2/(1+2s)},
\end{equation}
where $s=\frac{\alpha+1}n$.
Then, for any given $N>0$, by choosing 
\begin{equation}
\eps = (c/N)^{1+2s} |f|_{B^{\alpha}_{q,q}(\Omega)}^2,
\end{equation}
we would be able to guarantee a partition $P\in\conf(P_0)$ satisfying
$\#P\leq\#P_0+N$ and 
\begin{equation}
\osc(P)\lesssim\sum_{\tau\in P} e(\tau,P) \lesssim N^{-2s} |f|_{B^{\alpha}_{q,q}(\Omega)}^2.
\end{equation}

Let $\eps>0$.
We then recursively define $R_k=\{\tau\in P_k:e(\tau,P_k)>\eps\}$ and $P_{k+1}=\refine(P_k,R_k)$ for $k=0,1,\ldots$.
For all sufficiently large $k$ we will have $R_k=\varnothing$ since $|f|_{B^{\alpha}_{q,q}(\tau)} \lesssim |f|_{B^{\alpha}_{q,q}(\Omega)}$,
and $|\tau|$ is halved at each refinement.
Let $P=P_k$, where $k$ marks the first occurrence of $R_k=\varnothing$.
Since $e(\tau,P_k)\leq\eps$ for $\tau\in P_k$, \eqref{e:rho-bnd-pf} is immediate.

In order to get a bound on $\#P$, we estimate the cardinality of $R=R_0\cup R_1\cup\ldots\cup R_{k-1}$, and use \eqref{e:complete}.
Let $\Lambda_j=\{\tau\in R: 2^{-j-1}\leq|\tau|<2^{-j}\}$ for $j\in\Z$, and let $m_j=\#\Lambda_j$.
Note that the elements of $\Lambda_j$ (for any fixed $j$) are disjoint, since if any two elements intersect,
then they must come from different $R_k$'s as each $R_k$ consists of disjoint elements,
and hence the ratio between the measures of the two elements must lie outside $(\frac12,2)$.
This gives the trivial bound
\begin{equation}
m_j \leq 2^{j+1} |\Omega|.
\end{equation}
On the other hand, we have $e(\tau,P_k)>\eps$ for $\tau\in \Lambda_j$ with some $k$, which means
\begin{equation}
\eps < |\tau|^{2\delta} |f|_{B^{\alpha}_{q,q}(\tau)}^2
< 2^{-2j\delta} |f|_{B^{\alpha}_{q,q}(\tau)}^2 .
\end{equation}
Summing over $\tau\in\Lambda_j$, we get
\begin{equation}
m_j\eps^{q/2} 
\leq 2^{-jq\delta} \sum_{\tau\in \Lambda_j} |f|_{B^{\alpha}_{q,q}(\tau)}^q
\lesssim 2^{-jq\delta} |f|_{B^{\alpha}_{q,q}(\Omega)}^q,
\end{equation}
where we have used \eqref{e:Besov-sub-add}.
Finally, summing over $j$, we obtain
\begin{equation}
\#R
\leq \sum_{j=-\infty}^\infty m_j
\lesssim \sum_{j=-\infty}^\infty \min\left\{ 2^{j}, \eps^{-q/2} 2^{-jq\delta'} |f|_{B^{\alpha}_{q,q}(\Omega)}^q \right\}
\lesssim \eps^{-q/(2+2q\delta)} |f|_{B^{\alpha}_{q,q}(\Omega)}^{q/(1+q\delta)},
\end{equation}
which, in view of \eqref{e:complete} and $q/(1+q\delta) = 2/(1+2s)$, establishes the bound \eqref{e:card-P-bnd-pf}.
\end{proof}

In light of Lemma \ref{l:direct}, we immediately get the following corollary, which is one of our main results mentioned in the Introduction.

\begin{corollary}
Let $f\in B^\alpha_{q,q}(\Omega,\R^n)$ for some $0<q<\infty$ and $\alpha\geq\frac{n}q - \frac{n}2$ satisfying $0<\alpha<d-1+\max\{0,\frac1q-1\}$.
Then $(u,p)\in \tstA^s$ implies $(u,p)\in\tstA_*^s$, with $s=\frac{\alpha+1}n$.
\end{corollary}

\begin{remark}
By using the same arguments as above, one can get analogous results for general linear equations with constant coefficients.
Details of this investigation will be reported elsewhere.
\end{remark}

For completeness, we include the following result, which establishes a (one-sided) characterization of 
the standard approximation classes $\tstA^s$ in terms of Besov spaces.

\begin{theorem}\label{t:direct-H1}
We have $( B^{1+ns}_{q,q}(\Omega,\R^n) \cap V ) \times ( B^{ns}_{q,q}(\Omega) \cap Q ) \hookrightarrow\tstA^s$,
as long as $0<q<\infty$, $s>\frac1q-\frac12$, and $0<ns<d+\max\{0,\frac1q-1\}$.
\end{theorem}

\begin{proof}
In this proof, we assume that the Besov space seminorms related to the velocity variable are defined through $\omega_{d+1}$,
and that the seminorms related to the pressure variable are defined through $\omega_d$.
For $P\in\conf(P_0)$, let $Z_P:V\to V_P$ be a Scott-Zhang quasi-interpolation operator preserving the Dirichlet condition on $\partial\Omega$.
Let $u\in B^{1+ns}_{q,q}(\Omega,\R^n)\cap V$, and let $P\in\conf(P_0)$.
Then for any $\tau\in P$ and any $v\in(\Pol_d)^n$, we have
\begin{equation}
\begin{split}
\|u-Z_Pu\|_{H^1(\tau)} 
&\leq \|u-v\|_{H^1(\tau)} + \|Z_P(u-v)\|_{H^1(\tau)} \\
&\lesssim |\tau|^{-1/n} \|u-v\|_{L^2(\hat\tau)} + |u-v|_{H^1(\hat\tau)},
\end{split}
\end{equation}
where $\hat\tau$ is the star around $\tau$.
Now we shift to the reference situation where $\diam\,\hat\tau=1$.
By the embedding $B^{1+ns}_{q,q}\hookrightarrow H^1$, we can bound the last term as
\begin{equation}
|u-v|_{H^1(\hat\tau)} \lesssim \|u-v\|_{L^q(\hat\tau)} + |u-v|_{B^{1+ns}_{q,q}(\hat\tau)} = \|u-v\|_{L^q(\hat\tau)} + |u|_{B^{1+ns}_{q,q}(\hat\tau)} .
\end{equation}
For the other term, we have
\begin{equation}
\begin{split}
\|u-v\|_{L^2(\hat\tau)}
\lesssim \|u-v\|_{L^q(\hat\tau)} + |u|_{B^{1+ns}_{q,q}(\hat\tau)},
\end{split}
\end{equation}
this time using the embedding $B^{1+ns}_{q,q}\hookrightarrow L^2$.
Finally, Whitney's estimate gives
\begin{equation}
\|u-p\|_{L^q(\hat\tau)} 
\lesssim \omega_{d+1}(u,\hat\tau)_q
\lesssim |u|_{B^{1+ns}_{q,q}(\hat\tau)}.
\end{equation}
Leaving the reference situation by homogeneity, and combining the result, we have 
\begin{equation}
\|u-Z_Pu\|_{H^1(\Omega)}^2 
\lesssim \sum_{\tau\in P} |\tau|^{2\delta} |u|_{B^{1+ns}_{q,q}(\hat\tau)}^2 ,
\end{equation}
with $\delta=s+\frac12-\frac1q>0$.
Similarly, one can derive
\begin{equation}
\|p-\Pi_Pp\|_{L^2(\Omega)}^2 
\lesssim \sum_{\tau\in P} |\tau|^{2\delta} |p|_{B^{ns}_{q,q}(\tau)}^2 ,
\qquad
p\in B^{ns}_{q,q}(\Omega) \cap Q ,
\end{equation}
where $\Pi_P:Q\to Q_P$ is the $L^2$-orthogonal projector.
The rest of the proof can be completed in the same way as that of the preceding theorem.
\end{proof}

%%%%%%%%%%%%%%%%%%%%%%%%%%%%%%%%%%%%%%%%%%%%%%%%
\section*{Acknowledgements}

The author would like to thank Dirk Praetorius for an important comment regarding an earlier draft of this manuscript.
This work is supported by an NSERC Discovery Grant and an FQRNT Nouveaux Chercheurs Grant.

%%%%%%%%%%%%%%%%%%%%%%%%%%%%%%%%%%%%%%%%%%%%%%%%
%\bibliographystyle{dcu}
%\bibliographystyle{agsm}

%\bibliographystyle{nederlands}
%\bibliographystyle{kluwer}
%\bibliographystyle{jphysicsB}
%\bibliographystyle{jmr}

%\bibliographystyle{cell}
%\bibliographystyle{apalike}
\bibliographystyle{plainnat}
\bibliography{../bib/timur}

\end{document}